\documentclass[a4paper,12pt]{article}

\usepackage{ifthen}
\usepackage{amssymb}
\usepackage{amsmath}
\usepackage{MnSymbol}
\usepackage{stmaryrd}
\usepackage{mathrsfs}
\usepackage{theorem}
\usepackage[all]{xypic}
\usepackage{color}
\usepackage{graphicx}
\usepackage{xparse}
\usepackage[affil-it]{authblk}
\usepackage[backend = biber]{biblatex}
\makeatletter
\def\@maketitle{%
  \newpage
  \null
  \vskip 2em%
  \begin{center}%
  \let \footnote \thanks
    {\Large\bfseries \@title \par}%
    \vskip 1.5em%
    {\normalsize
      \lineskip .5em%
      \begin{tabular}[t]{c}%
        \@author
      \end{tabular}\par}%
    \vskip 1em%
    {\normalsize \@date}%
  \end{center}%
  \par
  \vskip 1.5em}
\makeatother

	{
		\theorembodyfont{\itshape}

		\newtheorem{theorem}{Theorem}[section]
		\newtheorem{lemma}[theorem]{Lemma}
		\newtheorem{proposition}[theorem]{Proposition}
		\newtheorem{corollary}[theorem]{Corollary}
	}

	{
		\theorembodyfont{\rmfamily}
		\newtheorem{definition}[theorem]{Definition}
		
		\newtheorem{remark}[theorem]{Remark}
		\newtheorem{question}[theorem]{Question}
	}

	\newenvironment{proof}{
		\goodbreak\par
		\textit{Proof.}%
	}{%
		\nopagebreak
		\hfill{\vrule width 1ex height 1ex depth 0ex}
		\medskip
		\goodbreak
	}

	\newcommand{\sizedescriptor}[2]
	{
		\ifthenelse{\equal{#1}{0}}{}{
		\ifthenelse{\equal{#1}{1}}{\big}{
		\ifthenelse{\equal{#1}{2}}{\Big}{
		\ifthenelse{\equal{#1}{3}}{\bigg}{
		\ifthenelse{\equal{#1}{4}}{\Bigg}{
		#2}}}}}
	}

	\newcommand{\proven}[1]{\underline{#1}\vspace{0.2em}\\}

	\newcommand{\ie}[1][~]{i.e.{#1}}

	\newcommand{\df}[1]{\emph{#1}}
	
	\newcommand{\equ}{\sim}
	\newcommand{\dfeq}{\mathrel{\mathop:}=}
	\newcommand{\dfeqrev}{=\mathrel{\mathop:}}
	
	\newcommand{\sepdfeq}{\quad\dfeq\quad}
	
	\newcommand{\impl}{\Rightarrow}

	\newcommand{\rstr}[1]{\left.{#1}\right|}
	\newcommand{\im}{\mathrm{im}}

	\newcommand{\exactlyone}[4][auto]{\exists!\, #2 \,{\in}\, #3\,.\sizedescriptor{#1}{\left}({#4}\sizedescriptor{#1}{\right})}

	\newcommand{\xsome}[3]{\exists\, #1 \,{\in}\, #2\,.\,#3}

	\newcommand{\st}[3][auto]{\sizedescriptor{#1}{\left}\{#2\;\sizedescriptor{#1}{\middle}|\;#3\sizedescriptor{#1}{\right}\}}

	\newcommand{\NN}{\mathbb{N}}

	\newcommand{\RR}{\mathbb{R}}

	\newdir{ >}{{}*!/-8pt/@{>}}
	\newdir{ (}{{}*!/-5pt/@^{(}}
	\newdir{|>}{!/5pt/@{|}*:(1,-.2)@{>}*:(1,+.2)@_{>}}

	\newcommand{\Sym}[1][n]{\mathrm{Sym}_{#1}}
	\newcommand{\Minsym}[1][n]{\mathrm{MinSym}_{#1}}
	\newcommand{\el}[2][n]{\mathtt{el}_{#1,#2}}
	\newcommand{\qsymm}{{\mathcal{S}}}
	\newcommand{\symm}{{\tilde{\mathcal{S}}}}

\title{Symmetric Polynomials in Tropical Algebra Semirings}

\date{}
\author{Sara Kali\v{s}nik}
\affil{Max Planck Institute for Mathematics in the Sciences, sara.kalisnik@mis.mpg.de}
\affil{Wesleyan University, skalisnikver@wesleyan.edu}

\author{Davorin Le\v{s}nik%
  \thanks{Electronic address: \texttt{davorin.lesnik@fmf.uni-lj.si}; This author was partially supported by the Air Force Office of Scientific Research, Air Force Materiel Command, USAF under Award No.~FA9550-14-1-0096.}}
\affil{Department of Mathematics, University of Ljubljana}

\begin{document}
	
\maketitle
\vspace{-1cm}
	\begin{abstract}
		The growth of tropical geometry has generated significant interest in the tropical semiring in the past decade. However, there are other semirings in tropical algebra that provide more information, such as the symmetrized $(\max, +)$, Izhakian's extended and  Izhakian-Rowen's supertropical semirings. In this paper we identify in which of these upper-bound semirings we can express symmetric polynomials in terms of elementary ones. We show that in the case of idempotent semirings we can do this precisely when the Frobenius property is satisfied, that in the case of supertropical semirings this is always possible, and that in non-trivial symmetrized semirings this is never possible. Our results allow us to determine the tropical algebra semirings where an analogue of the Fundamental Theorem of Symmetric Polynomials holds and to what extent.    
	\end{abstract}
	
	\section{Introduction}\label{SECTION: Introduction}
	Tropical algebra is a relatively new branch of mathematics, which has gained a lot of popularity over the last two decade~\cite{tropintro, speyersturm, Mikhalkin}. The adjective `tropical' was coined by French mathematicians in honor of the Brazilian computer scientist Imre Simon~\cite{Simon1988}, one of the pioneers in min-plus algebra. It builds on the older area more commonly known as max-plus algebra, which arises in semigroup theory, optimization, and computer science~\cite{maxlinear, gaubertmaxplus}. 

The tropical semiring lies at the heart of tropical geometry. In simplest terms, tropical geometry can be thought of as algebraic geometry over the tropical semiring, a piece-wise linear version of algebraic geometry, which replaces a variety by its combinatorial shadow. Although much work has been done, there is not yet a complete translation of the methods of algebraic geometry to the tropical situation. In particular, one of the main objects of study in algebraic geometry, invariant theory, has not been studied much in the tropical setting. 

In \cite{symtrop}, we initiated the translation of invariant theory by studying the tropical semiring. In this paper we build on our previous results and answer what happens in other semirings of interest to tropical algebraists, such as the symmetrized $(\max, +)$ semiring~\cite{gaubertmaxplus}, the extended tropical semiring~\cite{extendedsemiring}, the supertropical semiring~\cite{Izhakian20102222, extendedsemiring, Izhakian2011}. They are all \emph{upper-bound} semirings~\cite{Izhakian20112431, Izhakian201661}; therefore, we formulate the statements in this setting. 

One of the reasons all these different extensions of the tropical semiring were introduced was to help develop algorithms that depend only `on valuations.' For example, the symmetrization arises when considering the field of Puiseux series with real coefficients, equipped with the map which takes the valuation and the sign of a series. A signed element encodes the inverse image of a single element by this map, whereas a balanced element encodes all the series in a parameter $t$ with an asymptotic expansion $O(t^a)$. Similarly, the extended tropical semiring arises when considering the field of Puiseux series with complex coefficients, with an analogous interpretation. The symmetrized version has been used to design combinatorial algorithms over real nonarchimedean fields, working in generic cases, using only the sign and valuation~\cite{allamigeon2015tropicalizing}, and more recently in the work of X. Allamigeon et.~al.~\cite{allamigeon2017solving}.\footnote{There is a closely related alternative approach to the extensions of semirings, based on the notion of hyperfield coming back to Krasner. Both the extended tropical semiring and the symmetrized  $(\max, +)$  semiring can be identified with hyperfields, and conversely, any hyperfield can be identified with some powerset semiring. The interest in hyperfields arose recently after several papers on the topic appeared~\cite{connes2011hyperring, viro2010hyperfields, baker2016matroids}. }

	We study elementarity, \ie the ability to express symmetric polynomials with elementary ones in tropical semirings. Given $n \in \NN$, we say that a semiring $X$ is \df{$n$-elementary} when every symmetric polynomial $p$ in $n$ variables can be written as a polynomial in the elementary symmetric polynomials. The semiring $X$ is \df{fully elementary} when it is $n$-elementary for all $n \in \NN$.
	
 We prove that in upper-bound semirings $2$-elementarity is equivalent to the Frobenius property (Theorem~\ref{THEOREM: Frobenius and 2-elementarity}). In idempotent semirings the Frobenius property is equivalent to full elementarity (Theorem~\ref{THEOREM: Elementarity in idempotent semirings} and Corollary~\ref{COROLLARY: Elementarity in idempotent semirings}). In addition, supertropical semirings, including the extended tropical semiring, which are all Frobenius, are fully elementary (Theorem~\ref{THEOREM: Supertropical semirings are fully elementary}). As a corollary of these theorems we get that the tropical semiring $\RR_{\min}$ and the max-plus semiring $\RR_{\max}$ are fully elementary. Their symmetrizations $\symm(\RR_{\min})$ and $\symm(\RR_{\max})$, however, are not.	
		
	\section{Preliminaries}\label{SECTION: Preliminaries}
	
		Recall that $(X,+,0,\cdot)$ is a \df{semiring} when $(X,+,0)$ is a commutative monoid, $(X,\cdot)$ a semigroup, the multiplication $\cdot$ distributes over the addition $+$ and $0$ is an absorbing element, \ie $0 \cdot x = x \cdot 0 = 0$ for all $x \in X$.
				
		 A semiring is \df{unital} when it has the multiplicative unit $1$, \ie such an element $1\in X$ that $1\cdot x = x\cdot 1 =x$ for all $x\in X$.  A semiring is \df{commutative} when $\cdot$ is commutative. It is \df{idempotent} when $+$ is idempotent, \ie $x + x = x$ for all $x \in X$. A map between semirings is a semiring homomorphism when it preserves addition, zero and multiplication. If the semirings are unital, it is called a unital semiring homomorphism when it additionally preserves 1. For an excellent introduction to the theory of semirings, we refer the reader to \cite{semirings}.
		
		As usual, we often omit the $\cdot$ sign in algebraic expressions, and shorten the product of $n$ many factors $x$ to $x^n$. Also, we write just $X$ instead of $(X,+,0,\cdot)$ (or $(X,+,0,\cdot,1)$) when the operations are clear.
		
		Given a unital semiring $X$, we can view any natural number\footnote{In this paper we consider 0 to be a natural number. It represents a unit for addition in $X$.} $n \in \NN$ as an element of $X$ in the usual way:
		\[n = \underbrace{1 + 1 + \ldots + 1}_{\text{$n$-times}}.\] 
		However, this mapping (in fact, a unital semiring homomorphism) from $\NN$ to $X$ need not be injective; for example, if $X$ is idempotent, then $1 = 2$. In fact, that is a characterization of idempotency in unital semirings: we get the converse by multiplying the equality $1 = 2$ with an arbitrary $x \in X$.
		
		In any commutative monoid $(X,+,0)$ we can define a binary relation, \ie \emph{intrinsic ordering}, by
		\[
		a \leq b \textrm{ if and only if }\xsome{x}{X}{a + x = b}
		\]
		for $a, b \in X$.  
		
		The intrinsic order is reflexive (since $a + 0 = a$) and transitive (since if $a + x = b$ and $b + y = c$, then $c = b + y = a + x + y$). Thus, it is a preorder on $X$. Note that $0$ is a least element in this preorder (since $0 + a = a$), and $+$ is monotone, in the sense that if $a \leq b$ and $c \leq d$, then $a + c \leq b + d$. Of course, any semiring is a commutative monoid for $+$ and thus has the intrinsic order. Note that in this case multiplication is monotone as well: distributive laws give us that $a \leq b$ implies $a c \leq b c$ and $c a \leq c b$, and then it follows from $a \leq b$ and $c \leq d$ that $a c \leq b c \leq b d$.
		
		The point of the intrinsic order in this paper is the directedness it implies: for any $a, b \in X$ we can find an upper bound for them, namely $a + b$.  More generally, in the proofs we repeatedly use the fact that a part of a sum (with many summands) is in relation $\leq$ with the entirety of the sum.
		
		Recall that any preorder defines an equivalence relation by $a \approx b \dfeq a \leq b \land b \leq a$. A preorder is antisymmetric (hence, a partial order) when $\approx$ is equality. In general, a preorder on a set $X$ induces a partial order on the quotient set $X/_\approx$.
		
		The intrinsic order on a commutative monoid (or a semiring) is not necessarily antisymmetric; for example, in a group (or a ring) all elements are equivalent. Yet, antisymmetry is crucial in our arguments, hence the following definition, which already appeared in~\cite{Izhakian20112431, Izhakian201661}.
		
		\begin{definition}
			A semiring is \df{upper-bound} when its intrinsic order is antisymmetric (thus a partial order).
		\end{definition}
		
		A simple example of an upper-bound semiring is the set of natural numbers $\NN$, where the intrinsic order is the usual $\leq$. However, we will be particularly interested in idempotent semirings.
		
		\begin{proposition}\label{PROPOSITION: Order of idempotent semirings}
			Let $X$ be a semiring. Define a binary relation $\ll$ on $X$ by 
			\[
			a \ll b \textrm{ if and only if } a + b = b.
			\]
			Then $\ll$ is antisymmetric and transitive, and the following statements are equivalent.
			\begin{enumerate}
				\item
					$X$ is an idempotent semiring.
				\item
					The relation $\ll$ is reflexive (thus a partial order).
				\item
					The relation $\ll$ is the same as the intrinsic order $\leq$.
				\item
					For any $a, b \in X$ their sum $a + b$ is the unique join (the least upper bound) of $a$ and $b$ in the intrinsic order.
			\end{enumerate}
			In particular, any idempotent semiring is an upper-bound semiring.
		\end{proposition}
		\begin{proof}
		See Example 5.3~\cite{Izhakian201661}.
		\end{proof}
			
					
		
		We now turn our attention to polynomials over semirings.
		
		\begin{definition}\label{DEFINITION: Polynomials in semirings}
			Let $X$ be a unital commutative semiring.
			\begin{itemize}
				\item
					Let $m, n, d_{1,1}, \ldots, d_{m,n}$ be natural numbers and $a_1, \ldots, a_m \in X$. A \df{polynomial} is a syntactic object of the form $\sum_{k = 1}^m a_k \prod_{j = 1}^n x_j^{d_{k,j}}$ (an individual summand $a_k \prod_{j = 1}^n x_j^{d_{k,j}}$ is called a \df{monomial}, and if its coefficient $a_k$ is equal to $1$, a \df{pure monomial}). The algebra of all polynomials in variables $x_1, \ldots, x_n$ over $X$ is denoted by $X[x_1, \ldots, x_n]$.
				\item
					Each polynomial has its corresponding \df{polynomial function}. A polynomial function is a function in the image of the algebra homomorphism
					\[X[x_1, \ldots, x_n] \to \{\text{functions } X^n \to X\}\]
					which takes a syntactic polynomial expression and assigns to it the function it represents.
				\item
					A polynomial is \df{symmetric} when for each monomial $a_k \prod_{j = 1}^n x_j^{d_{k,j}}$ in it and each permutation $\sigma \in S_n$ the monomial $a_k \prod_{j = 1}^n x_{\sigma(j)}^{d_{k,j}}$ also appears in it, up to a change of the order of factors.\footnote{That is, we consider symmetry relative to commutativity of multiplication. For example, the polynomial expression $x y$ is symmetric: the transposition of variables gives $y x$, which we identify with $x y$. Of course, it would not make sense to consider $x y$ symmetric over a non-commutative semiring.} A polynomial function is \df{symmetric} when it can be represented by a symmetric polynomial.
				\item
					For any $n \in \NN$ and $j \in \{1, \ldots, n\}$ the \df{elementary symmetric polynomial} $\el{j} \in X[x_1, \ldots, x_n]$ is the sum of all products of $j$ different variables, \ie
					\begin{align*}
						\el{1}(x_1, \ldots, x_n) &= x_1 + x_2 + \ldots + x_n,\\
						\el{2}(x_1, \ldots, x_n) &= x_1 x_2 + x_1 x_3 + x_2 x_3 + \ldots + x_{n-1} x_n,\\
						\vdots\\
						\el{n}(x_1, \ldots, x_n) &= x_1 x_2 \ldots x_n.
					\end{align*}
					Note that $\el{j}$ has $\binom{n}{j}$ terms. We will use the same notation also for the corresponding polynomial functions $\el{j}\colon X^n \to X$.
			\end{itemize}
		\end{definition}
		
		\begin{remark}\label{REMARK: variants of symmetry}
			There are two reasonable definitions of when a polynomial function is symmetric: if it is represented by a symmetric polynomial (let us say that a function is `syntactically symmetric' in this case), or if its values are invariant under arbitrary permutations of variables, that is, $p(x_1, \ldots, x_n) = p(x_{\sigma(1)}, \ldots, x_{\sigma(1)})$ for all $(x_1, \ldots, x_n)$ and permutations $\sigma \in S_n$ (say that $p$ is `semantically symmetric'). Every syntactically symmetric polynomial function is clearly semantically symmetric. In Remark~\ref{REMARK: symmetry over idempotent semirings} we note that for polynomial functions over idempotent unital commutative semirings the converse also holds. We do not know whether the converse is true over an arbitrary unital commutative semiring, and we pose this as a part of Question~\ref{QUESTION: match between syntactic and semantic symmetry}. For the purposes of our theorems, a `symmetric polynomial function' refers to syntactic symmetry (as stated in Definition~\ref{DEFINITION: Polynomials in semirings}), since in proofs we use polynomials as syntactic objects. 
		\end{remark}
		
		The goal of this paper is to determine when it is possible to express symmetric polynomial functions in terms of elementary symmetric polynomials in certain upper-bound semirings. For the sake of neatly expressing this, we introduce the following definition.
		
		\begin{definition}
			Let $X$ be a unital commutative semiring.
			\begin{itemize}
				\item
					Given $n \in \NN$, $X$ is \df{$n$-elementary} when for every symmetric polynomial function $p$ in $n$ variables there exists a polynomial function $r$ in $n$ variables, such that \[
					p(x_1, \ldots, x_n) = r\big(\el{1}(x_1, \ldots, x_n), \ldots, \el{n}(x_1, \ldots, x_n)\big)\] for all $x_1, \ldots, x_n \in X$.
				\item
					$X$ is \df{fully elementary} when it is $n$-elementary for all $n \in \NN$.
			\end{itemize}
		\end{definition}
		
		Any unital commutative semiring is $0$-elementary and $1$-elementary.
		
		In the following four sections we discuss elementarity in different semirings.

	\section{Elementarity and Frobenius Equalities}\label{SECTION: Elementarity and Frobenius Equalities}
	
		As we shall see, elementarity in semirings is closely related to Frobenius equalities. Recall that the \df{Frobenius equality}\footnote{Also called `Freshman's Dream' for obvious reasons.} for $n \in \NN$ in a unital commutative semiring $X$ states that $(x + y)^n = x^n + y^n$ for all $x, y \in X$.
		
		Note that the Frobenius equality for $0$ is a bit special: it states that $1 = 1 + 1$, so it is equivalent to the semiring being idempotent. For other Frobenius equalities we have the following definition.
		
		\begin{definition}
			A unital commutative semiring $X$ is \df{Frobenius} when it satisfies Frobenius equalities for all $n \in \NN_{\geq 1}$.
		\end{definition}
		
		Here are some sources of Frobenius semirings.
		
		\begin{proposition}\label{PROPOSITION: Linear idempotent semiring is Frobenius}
			Let $X$ be an idempotent unital commutative semiring, in which the intrinsic order is linear, \ie $x \leq y$ or $y \leq x$ for all $x, y \in X$. Then $X$ is Frobenius.
		\end{proposition}
		
		\begin{proof}
			Let $n \in \NN_{\geq 1}$ and $x, y \in X$. If $x \leq y$, then by monotonicity of multiplication  $x^n \leq y^n$. By Proposition~\ref{PROPOSITION: Order of idempotent semirings}
			\[(x + y)^n = y^n = x^n + y^n.\]
			Similarly for $y \leq x$.
		\end{proof}
		
		Not every Frobenius idempotent unital commutative semiring is linearly ordered. For example, a product of any two Frobenius idempotent unital commutative semirings is again a Frobenius idempotent unital commutative semiring, but is not linearly ordered if both factors are non-trivial.
		
		Any idempotent unital commutative semiring $X$ which is \df{multiplicatively cancellable}, in the sense that $a \cdot x = b \cdot x \implies a = b$ for all $x \in X \setminus \{0\}$, is also Frobenius --- see~\cite[proof of Lemma~4.3]{trove.nla.gov.au/work/38992473}.
		
		We already mentioned that the Frobenius equality for $0$ amounts to the idempotency of the semiring, which can be expressed as $1 = 2$ (and consequently $m = n$ in $X$ for all $m, n \in \NN_{\geq 1}$, since we can keep adding $1$ to both sides of the equation). Other Frobenius equalities also give us some equalities between natural numbers in a semiring.
		
		\begin{lemma}\label{LEMMA: 2 in Frobenius semirings}
			Let $X$ be a Frobenius semiring.
			\begin{enumerate}
				\item
					Then $2 = 4$ in $X$. Consequently, any $m, n \in \NN_{\geq 2}$ are the same in $X$ if they have equal parity.
				\item
					If $X$ is upper-bound, then $2 = 3$ in $X$. Consequently, any $m, n \in \NN_{\geq 2}$ are the same in $X$.
			\end{enumerate}
		\end{lemma}
		
		\begin{proof}
			\begin{enumerate}
				\item
					We have $2 = 1^2 + 1^2 = (1 + 1)^2 = 4$. It follows inductively that $2 = 2k$ in $X$ for all $k \in \NN_{\geq 1}$.
				\item
					By definition of $\leq$ we have $2 \leq 2 + 1 = 3 \leq 3 + 1 = 4 = 2$. By antisymmetry of $\leq$ we get $2 = 3$. It follows inductively that $2 = k$ in $X$ for all $k \in \NN_{\geq 2}$.
			\end{enumerate}
		\end{proof}
		
		\begin{remark}
			If a Frobenius semiring is not upper-bound, we might not have $2 = 3$. For example, take the semiring $\NN$ and define $a, b \in \NN$ to be equivalent when they are equal or they are both $\geq 2$ and of equal parity. The quotient $\NN/_\equ = \{[0], [1], [2], [3]\}$ inherits the semiring structure, for which it is Frobenius, and we have $[2] \neq [3]$. Of course, $\NN/_\equ$ is then not upper-bound as $[0] \leq [1] \leq [2] \leq [3] \leq [2]$.
		\end{remark}
		
		The Frobenius property allows the following partial characterization of elementarity.
		
		\begin{theorem}\label{THEOREM: Frobenius and 2-elementarity}
			Let $X$ be a unital commutative semiring.
			\begin{enumerate}
				\item
					If $X$ is Frobenius, it is $2$-elementary.
				\item \label{THEOREM: Frobenius and 2-elementarity: 2}
					If $X$ is upper-bound, the converse also holds.
			\end{enumerate}
		\end{theorem}
		
		\begin{proof}
			\begin{enumerate}
				\item
					Let $p(x,y) = \sum_{k = 1}^m a_k x^{i_k} y^{j_k}$ be a symmetric polynomial. Hence, for any monomial $a_k x^{i_k} y^{j_k}$ in $p$, if $i_k \neq j_k$, the polynomial also possesses the monomial of the form $a_k x^{j_k} y^{i_k}$. Thus we can write
					\[p(x,y) = \sum_{k \in \{1, \ldots, m\}, i_k = j_k} a_k (x y)^{i_k} + \sum_{k \in \{1, \ldots, m\}, i_k > j_k} a_k (x^{i_k} y^{j_k} + x^{j_k} y^{i_k}) =\]
					\[= \sum_{k \in \{1, \ldots, m\}, i_k = j_k} a_k (x y)^{i_k} + \sum_{k \in \{1, \ldots, m\}, i_k > j_k} a_k (x y)^{j_k} (x^{i_k - j_k} + y^{i_k - j_k}).\]
					Since $i_k - j_k \geq 1$ in the last sum, Frobenius equality implies 
					\[
					x^{i_k - j_k} + y^{i_k - j_k} = (x + y)^{i_k - j_k}.
					\]
					Therefore
					\[p(x,y) = \sum_{k \in \{1, \ldots, m\}, i_k = j_k} a_k \el[2]{2}^{i_k}(x,y) + \sum_{k \in \{1, \ldots, m\}, i_k > j_k} a_k \el[2]{2}^{j_k}(x,y) \el[2]{1}^{i_k - j_k}(x,y).\]
					We can simplify this to
					\[p(x,y) = \sum_{k \in \{1, \ldots, m\}, i_k \geq j_k} a_k \el[2]{2}^{j_k}(x,y) \el[2]{1}^{i_k - j_k}(x,y).\]
				\item
					Take any $n \in \NN_{\geq 1}$. The polynomial $x^n + y^n$ is symmetric, so by assumption we can write
					\[x^n + y^n =   \sum_{k = 1}^m a_k \el[2]{2}^{i_k}(x,y) \el[2]{1}^{j_k}(x,y) = \sum_{k = 1}^m a_k (x y)^{i_k} (x + y)^{j_k}.\]
					This holds for all $x, y \in X$. Setting $y$ to $0$ and replacing $x$ with $x + y$ yields $(x + y)^n = \sum_{k \in \{1, \ldots, m\}, i_k = 0} a_k (x + y)^{j_k}$. Since adding summands can only increase the value in the intrinsic order, 
					\[x^n + y^n \leq \sum_{k = 0}^n \binom{n}{k} x^{n-k} y^k = (x + y)^n = \sum_{k \in \{1, \ldots, m\}, i_k = 0} a_k (x + y)^{j_k} \leq\]
					\[\leq \sum_{k = 1}^m a_k (x y)^{i_k} (x + y)^{j_k} = x^n + y^n\]
					(note that we needed $n \geq 1$ for the first step in this chain). By antisymmetry of $\leq$ we conclude that $x^n + y^n = (x + y)^n$.
			\end{enumerate}
		\end{proof}
		
		\begin{remark}
			In part~\ref{THEOREM: Frobenius and 2-elementarity: 2} of Theorem~\ref{THEOREM: Frobenius and 2-elementarity} the assumption of $X$ being upper-bound is necessary. For example, symmetric polynomial functions over $\RR$ can be written as polynomials of elementary symmetric ones, but Frobenius equalities do not hold.
		\end{remark}

	\section{Elementarity in Idempotent Semirings}\label{Elementarity in Idempotent Semirings}
		
		We proved that in upper-bound semirings the Frobenius property is equivalent to $2$-elementarity. We now turn our attention to idempotent semirings and prove that in this case the Frobenius property is equivalent to full elementarity. In~\cite{symtrop} we proved full elementarity for the tropical semiring, $\RR_{\min}$; what follows is an adaptation of that proof that works for general idempotent Frobenius semirings.
		
		Let $p$ be a polynomial in $n \in \NN$ variables. Define
		\[\Sym(p)(x_1, \ldots, x_n) \dfeq \sum_{\sigma \in S_n} p(x_{\sigma(1)}, \ldots, x_{\sigma(n)}).\]
		This induces symmetrization also on the level of polynomial functions, for which we use the same notation.
		
		\begin{proposition}\label{PROPOSITION: Basic symmetrization facts}
			Let $X$ be an idempotent unital commutative semiring, $n \in \NN$, $a \in X$ and let $p, q \in X[x_1, \ldots, x_n]$ be arbitrary polynomials.
			\begin{enumerate}
				\item
					$\Sym(p)$ is a symmetric polynomial.
				\item
					A polynomial $p$ is symmetric if and only if $p = \Sym(p)$.
				\item
					$\Sym(p + q) = \Sym(p) + \Sym(q)$.
				\item
					$\Sym(a \cdot p) = a \cdot \Sym(p)$.
				\item
					$\el{j} = \Sym(x_1 \ldots x_j)$ for all $j \in \NN$, $1 \leq j \leq n$.
			\end{enumerate}
		\end{proposition}
		
		\begin{proof}
			We leave the proof to the reader.
		\end{proof}
		
		\begin{remark}
			In general the \df{symmetrization} of a polynomial is defined as
			\[\Sym(p)(x_1, \ldots, x_n) \dfeq \frac{1}{n!} \sum_{\sigma \in S_n} p(x_{\sigma(1)}, \ldots, x_{\sigma(n)}),\]
			\ie as the \emph{average} over permutations, so that we have $p = \Sym(p)$ for a symmetric $p$. Of course, this is only well defined when factorials are invertible (equivalently, when positive natural numbers are invertible) in the semiring. That is not a problem in an idempotent semiring though, as we have $1 = 2 = 3 = \ldots$, and the definition of $\Sym$ reduces to just the sum over permutations.
		\end{remark}
				
		\begin{remark}\label{REMARK: symmetry over idempotent semirings}
			Over idempotent unital commutative semirings the notions of syntactically symmetric and semantically symmetric polynomial functions coincide (recall Remark~\ref{REMARK: variants of symmetry}). This is because for a polynomial function $p$ the equality $p = \Sym(p)$ is equivalent to both $p$ being syntactically symmetric and semantically symmetric.
		\end{remark}
		
		The following lemma will be used as an inductive step, with $j$ being the variable, for which we do the induction.
		
		\begin{lemma}\label{LEMMA: Factoring elementary symmetric polynomials}
			Let $X$ be an idempotent Frobenius semiring and $n \in \NN_{\geq 1}$. Then for all $j \in \NN$ with $n \geq j \geq 1$ and $d_1, \ldots, d_j \in \NN$ with $d_1 \geq d_2 \geq \ldots \geq d_j$
			\[\Sym(x_1^{d_1} \ldots x_j^{d_j}) = \el{j}^{d_j} \cdot \Sym(x_1^{d_1 - d_j} \ldots x_{j-1}^{d_{j-1} - d_j})\]
			at the level of polynomial functions.
		\end{lemma}
		
		\begin{proof}
			Clearly, the equality holds for $d_j = 0$, so assume hereafter that $d_j \geq 1$. We are trying to prove that
			\[
			\sum_{\sigma \in S_n} x_{\sigma(1)}^{d_1} \ldots x_{\sigma(j)}^{d_j} = \Big(\sum_{\pi \in S_n} x_{\pi(1)} \ldots x_{\pi(j)}\Big)^{d_j} \cdot \sum_{\rho \in S_n} x_{\rho(1)}^{d_1 - d_j} \ldots x_{\rho(j-1)}^{d_{j-1} - d_j},
			\]
			which by Frobenius property reduces to
			\[
			\sum_{\sigma \in S_n} x_{\sigma(1)}^{d_1} \ldots x_{\sigma(j)}^{d_j} = \sum_{\pi \in S_n} \big(x_{\pi(1)} \ldots x_{\pi(j)}\big)^{d_j} \cdot \sum_{\rho \in S_n} x_{\rho(1)}^{d_1 - d_j} \ldots x_{\rho(j-1)}^{d_{j-1} - d_j}.
			\]
			An idempotent semiring is upper-bound, so it suffices to prove inequality in both directions.
			
			For any permutation $\sigma \in S_n$, $x_{\sigma(1)}^{d_1} \ldots x_{\sigma(j)}^{d_j} = (x_{\sigma(1)} \ldots x_{\sigma(j)})^{d_j} \cdot x_{\sigma(1)}^{d_1 - d_j} \ldots x_{\sigma(j-1)}^{d_{j-1} - d_j}$, so every summand from the left-hand side also appears on the right-hand side. Thus
			\[\Sym(x_1^{d_1} \ldots x_j^{d_j}) \leq \el{j}^{d_j} \cdot \Sym(x_1^{d_1 - d_j} \ldots x_{j-1}^{d_{j-1} - d_j}).\]
			
			Conversely, take any $\pi, \rho \in S_n$, and consider the summand 
			\[
			s \dfeq (x_{\pi(1)} \ldots x_{\pi(j)})^{d_j} \cdot x_{\rho(1)}^{d_1 - d_j} \ldots x_{\rho(j-1)}^{d_{j-1} - d_j}
			\] from the right-hand side. Some of the variables might appear in both parts of this product; denote $I \dfeq \st{i \in \{1, \ldots, j\}}{\xsome{k}{\{1, \ldots, j\}}{\rho(i) = \pi(k)}}$. Since we can arbitrarily permute the variables in the product $x_{\pi(1)} \ldots x_{\pi(j)}$ without changing its value, we may assume without loss of generality that $\pi(i) = \rho(i)$ for all $i \in I$. Denote $J \dfeq \{1, \ldots, j\} \setminus I$; then for any $i \in J$ (taking into account Frobenius)
			\[x_{\pi(i)}^{d_j} \cdot x_{\rho(i)}^{d_i - d_j} \leq \sum_{k = 0}^{d_i} \binom{d_i}{k} x_{\pi(i)}^{k} x_{\rho(i)}^{d_i - k} = \big(x_{\pi(i)} + x_{\rho(i)}\big)^{d_i} = x_{\pi(i)}^{d_i} + x_{\rho(i)}^{d_i},\]
			so
			\[s \leq \prod_{i \in I} x_{\pi(i)}^{d_i} \cdot \prod_{i \in J} (x_{\pi(i)}^{d_i} + x_{\rho(i)}^{d_i}).\]
			If we use distributivity to fully expand this product, we see that each summand we get also appears in $\sum_{\sigma \in S_n} x_{\sigma(1)}^{d_1} \ldots x_{\sigma(j)}^{d_j}$. Since $+$ is supremum in an idempotent semiring, we conclude that
			\[\Sym(x_1^{d_1} \ldots x_j^{d_j}) \geq \el{j}^{d_j} \cdot \Sym(x_1^{d_1 - d_j} \ldots x_{j-1}^{d_{j-1} - d_j}).\]
		\end{proof}
		
		As stated in the lemma, the equality holds on the level of functions, but it does not hold on the level of polynomials. For example, we have $\Sym[2](x_1^2) = x_1^2 + x_2^2$, whereas $\el[2]{1}^2 \cdot \Sym[2](x_1^0) = (x_1 + x_2)^2 \cdot 1 = x_1^2 + 2 x_1 x_2 + x_2^2$. The same example shows also that the following lemma does not hold on the level of polynomials.
		
		\begin{lemma}\label{LEMMA: Elementary factorization of a monomial}
			Let $X$ be an idempotent Frobenius semiring. Then the symmetrization of any pure monomial (\ie monomial with coefficient $1$) is at the level of functions equal to a product of elementary symmetric polynomials.
		\end{lemma}
		
		\begin{proof}
			Since $\Sym(x_1^{d_1} \ldots x_n^{d_n}) = \Sym(x_{\sigma(1)}^{d_1} \ldots x_{\sigma(n)}^{d_n})$ for any permutation $\sigma \in S_n$, any symmetrization of a pure monomial in $n \in \NN$ variables can be written as $\Sym(x_1^{d_1} \ldots x_n^{d_n})$ where $d_1 \geq d_2 \geq \ldots \geq d_n$. Using Lemma~\ref{LEMMA: Factoring elementary symmetric polynomials} as the inductive step, we then see that
			\[\Sym(x_1^{d_1} \ldots x_n^{d_n}) = \el{n}^{d_n} \cdot \el{n-1}^{d_{n-1} - d_n} \cdot \el{n-2}^{d_{n-2} - d_{n-1}} \cdot \ldots \cdot \el{1}^{d_1 - d_2}.\]
		\end{proof}
		
		A semiring $X$ is \df{$2$-cancellative}\footnote{We use this name because in unital semirings we can rewrite the condition as $2 x = 2 y \impl x = y$, \ie we can cancel $2$.} when $x+x = y+y \impl x = y$ holds for all $x, y \in X$. Any idempotent semiring is $2$-cancellative.
		
		\begin{theorem}\label{THEOREM: Elementarity in idempotent semirings}
			The following statements are equivalent for any unital commutative semiring $X$.
			\begin{enumerate}
				\item
					$X$ is $2$-elementary, upper-bound and $2$-cancellative.
				\item
					$X$ is Frobenius and idempotent.
			\end{enumerate}
			Furthermore, when these statements are satisfied, $X$ is fully elementary.
		\end{theorem}
		
		\begin{proof}
			\begin{itemize}
				\item\proven{$(1 \impl 2)$}
					By Theorem~\ref{THEOREM: Frobenius and 2-elementarity} $X$ is Frobenius. By Lemma~\ref{LEMMA: 2 in Frobenius semirings} we have $2 = 4$ in $X$, that is, $2 \cdot 1 = 2 \cdot 2$. Cancel $2$ to get $1 = 2$, which is idempotency.
				\item\proven{$(2 \impl 1 \land \text{$X$ is fully elementary})$}
					Any idempotent semiring is upper-bound and $2$-cancellative. Now take any symmetric polynomial $p(x_1, \ldots, x_n) = \sum_{k = 1}^m a_k \prod_{j = 1}^n x_j^{d_{k,j}}$. Using Proposition~\ref{PROPOSITION: Basic symmetrization facts} we get (at the level of functions)
					\[
					p(x_1, \ldots, x_n) = \Sym(p)(x_1, \ldots, x_n) = \sum_{k = 1}^m a_k\, \Sym\big(\prod_{j = 1}^n x_j^{d_{k,j}}\big).
					\]
					By Lemma~\ref{LEMMA: Elementary factorization of a monomial} each $\Sym\big(\prod_{j = 1}^n x_j^{d_{k,j}}\big)$ is a product of elementary symmetric polynomials, so $p$ can also be expressed as a polynomial in elementary symmetric polynomials.
			\end{itemize}
		\end{proof}
		
		\begin{corollary}\label{COROLLARY: Elementarity in idempotent semirings}
			An idempotent unital commutative semiring is fully elementary if and only if it is Frobenius.
		\end{corollary}
		
		\begin{proof}
			The claim follows from Theorem~\ref{THEOREM: Elementarity in idempotent semirings}, since any idempotent semiring is upper-bound and $2$-cancellative.
		\end{proof}

	\section{Elementarity in Supertropical Semirings}\label{SECTION: Elementarity in Supertropical Semirings}
	
		Another way to phrase Theorem~\ref{THEOREM: Elementarity in idempotent semirings} is that as soon as an upper-bound semiring is $2$-cancellative, full elementarity is equivalent to idempotency and Frobenius. This gives us a characterization of full elementarity, but not a fully satisfactory one since there is an important class of upper-bound semirings which in general are not $2$-cancellative: supertropical semirings~\cite{Izhakian20102222, extendedsemiring, Izhakian2011, Izhakian201661}. In this section we prove that supertropical semirings are fully elementary.
		
		The theory of supertropical semirings started with the so-called \df{extended tropical semiring} which is obtained by starting with the tropical semiring, adding a `ghost copy' of it and identifying their least elements (see~\cite{extendedsemiring} for details). Supertropical semirings are what we obtain when we abstract certain useful properties of this extended tropical semiring.  
		
		Let us now formally define these semirings. We start by recalling the relevant definitions. For any semiring $X$ define its \df{ghost}~\cite{Izhakian20102222} as
		\[\nu{X} \dfeq \st{a \in X}{a = a + a}.\]
		The elements in $\nu{X}$ are called \df{ghost elements} of $X$ and the elements in $X \setminus \nu{X}$ are \df{tangible}. Clearly $0 \in \nu{X}$ and if $a, b \in \nu{X}$, then $a + b \in \nu{X}$. Also, if $a \in \nu{X}$ and $x \in X$, then $x \cdot a \in \nu{X}$ and $a \cdot x \in \nu{X}$. In short, $\nu{X}$ is a semiring ideal in $X$. We can now immediately conclude from the definition that $\nu{X}$ is an idempotent semiring.
		
		Define the map $\nu\colon X \to X$ by $\nu(x) \dfeq x + x$. Clearly, $\nu$ is an additive monoid homomorphism, although it is not necessarily a semiring homomorphism (take for example $X = \NN$). Also, $\nu$ is a monotone map (with regard to the intrinsic order) and we have $x \leq \nu(x)$ for all $x \in X$.
		
		By definition, $\nu{X}$ is the set of fixed points of $\nu$. Note that if $X$ is unital (so we have natural numbers in $X$), we can write $\nu(x) = 2x$ and $\nu{X} = \st{a \in X}{a = 2a}$.
		
		\begin{proposition}\label{PROPOSITION: Basic properties of nu}
			Let $X$ be a unital semiring. The following statements are equivalent.
			\begin{enumerate}
				\item
					$2 = 4$ in $X$.
				\item
					$\nu$ is a semiring homomorphism.
				\item
					$\nu$ is a projection (\ie $\nu \circ \nu = \nu$).
				\item
					The image of $\nu$ is $\nu{X}$. In particular, the corestriction $\rstr{\nu}^{\nu{X}}\colon X \to \nu{X}$ exists.
			\end{enumerate}
			If these statements hold, then $\nu{X}$ is a unital (idempotent) semiring with the multiplicative unit $\nu(1)$, so $\rstr{\nu}^{\nu{X}}$ is a unital semiring homomorphism.
		\end{proposition}
		
		\begin{proof}
			\begin{itemize}
				\item\proven{$(1 \impl 2)$}
					We know that $\nu$ is an additive monoid homomorphism. It remains to check that $\nu$ preserves multiplication: $\nu(x) \cdot \nu(y) = 2x \cdot 2y = 4 x y = 2 x y = \nu(x y)$.
				\item\proven{$(2 \impl 1)$}
					$2 = \nu(1) = \nu(1 \cdot 1) = \nu(1) \cdot \nu(1) = 2 \cdot 2 = 4$.
				\item\proven{$(1 \impl 3)$}
					$\nu(\nu(x)) = 2 \cdot 2x = 4x = 2x = \nu(x)$.
				\item\proven{$(3 \impl 4)$}
					Since $\nu{X}$ is the set of fixed points of $\nu$, we always have $\nu{X} \subseteq \im(\nu)$. For the reverse inclusion, take any $x \in X$. Then $2 \cdot \nu(x) = \nu(\nu(x)) = \nu(x)$, so $\nu(x) \in \nu{X}$.
				\item\proven{$(4 \impl 1)$}
					Since $2 = \nu(1) \in \im(\nu) = \nu{X}$, we get $2 = 2 + 2 = 4$.
			\end{itemize}
			
			Assume now that the given equivalent statements hold. Then for any $x \in X$ we have $\nu(1) \cdot \nu(x) = \nu(1 \cdot x) = \nu(x)$, and likewise $\nu(x) \cdot \nu(1) = \nu(x)$, so $\nu(1)$ is indeed the multiplicative unit in $\nu{X}$.
		\end{proof}
		
		We recall the definition from~\cite{Izhakian201661}.
		
		\begin{definition}
			A \df{supertropical semiring} is a unital commutative semiring which satisfies the following:
			\begin{itemize}
				\item
					the equivalent statements from Proposition~\ref{PROPOSITION: Basic properties of nu},
				\item
					for all $a, b \in X$ with $\nu(a) \neq \nu(b)$ we have $a + b \in \{a, b\}$,
				\item
					for all $a, b \in X$ with $\nu(a) = \nu(b)$ we have $a + b = \nu(a) = \nu(b)$.
			\end{itemize}
		\end{definition}
		
		\begin{proposition}\label{PROPOSITION: basic properties of supertropical semirings}
			The following holds for any supertropical semiring $X$.
			\begin{enumerate}
				\item
					$\nu{X}$ is \df{bipotent} (\ie for all $a, b \in \nu{X}$ we have $a + b \in \{a, b\}$), and therefore a linearly ordered idempotent semiring.
				\item
					$X$ is upper-bound. The restriction of $\leq$ from $X$ to $\nu{X}$ matches the intrinsic order on $\nu{X}$.
				\item
					$X$ is Frobenius and $2 = 3$ in $X$.
			\end{enumerate}
		\end{proposition}
		
		\begin{proof}
			\begin{enumerate}
				\item
					$\nu{X}$ is bipotent by~\cite{Izhakian20112431}. Hence for any elements $a, b \in \nu{X}$ one of them is their common upper bound $a + b$, so they must be comparable.
				\item
					See~\cite{Izhakian201661}(Proposition~5.7).
				\item
					By~\cite{Izhakian20102222}(Proposition~3.7) and Lemma~\ref{LEMMA: 2 in Frobenius semirings}.
			\end{enumerate}
		\end{proof}
		
		In any semiring we define the \df{strict order relation} $<$ in the expected way: $a < b$ means $a \leq b$ and $a \neq b$. The relation $<$ is irreflexive, asymmetric and transitive. If the intrinsic order $\leq$ is a linear partial order, then $<$ satisfies the law of trichotomy: for any $a, b \in X$ exactly one of the statements $a < b$, $a = b$, $b < a$ holds.
		
		We examine some properties of fibers of $\nu$ in supertropical semirings.
		
		\begin{lemma}\label{LEMMA: Fibers of nu}
			Let $X$ be a supertropical semiring and $a \in X$.
			\begin{enumerate}
				\item
					The fiber $\nu^{-1}(a)$ is non-empty if and only if $a \in \nu{X}$. In this case $a$ is the largest element (with regard to the intrinsic order $\leq$) in $\nu^{-1}(a)$.
				\item
					The elements in $\nu^{-1}(a) \setminus \{a\}$ are incomparable.
			\end{enumerate}
		\end{lemma}
		
		\begin{proof}
			\begin{enumerate}
				\item
					Straightforward.
				\item
					Suppose we have $x, y \in \nu^{-1}(a) \setminus \{a\}$ with $x \leq y$ and $x \neq y$. Then $x < y < a$. There exists such $u \in X$ that $x + u = y$, in particular $u \leq y$. It follows that $2u \leq 2y = a$. We cannot have $2u = a$, as that implies $y = x + u = a$, a contradiction. Thus $2u < a = 2x$, so $x = x + u = y$, another contradiction.
			\end{enumerate}
		\end{proof}
		
		In summary, fibers of $\nu$ look like this: \raisebox{2.5ex}{$\xymatrix@-3ex{& \circ &\\ \bullet \ar[ru] & \bullet \ar[u] & \bullet \ar[lu]}$} --- that is, a bunch of (possibly zero) incomparable (tangible) elements, with a ghost element on the top. The entire supertropical semiring is then a disjoint union of such fibers, with the ghost part linearly ordered.
		
		\begin{lemma}\label{LEMMA: Strict inequality in supertropical semirings}
			Let $X$ be a supertropical semiring and $a, b, x, y \in X$.
			\begin{enumerate}
				\item
					We have
					\[a + b = b \iff a < b \lor (a \leq b \land b \in \nu{X}).\]
					In particular, if $a < b$, then $a + b = b$.
				\item
					If $b \cdot x$ is tangible, then
					\[a < b \iff a \cdot x < b \cdot x.\]
				\item
					If $b \cdot y$ is tangible, $a < b$ and $x \leq y$, then $a \cdot x < b \cdot y$.
				\item
					If $y$ is tangible, we have
					\[x < y \iff 2x < y \iff 2x < 2y.\]
				\item \label{LEMMA: Strict inequality in supertropical semirings:5}
					If $x$ and $y$ are tangible, exactly one of the following holds: $x < y$, $y < x$ or $x$ and $y$ are in the same fiber of $\nu$.
			\end{enumerate}
		\end{lemma}
		
		\begin{proof}
			\begin{enumerate}
				\item
					Suppose $b = a + b \geq a$. If $a \neq b$, then $a < b$. If $a = b$, we have $b = a + b = 2b$.
					
					Conversely, suppose $a \leq b$ and $b = 2b$. Then $b \leq a + b \leq 2b = b$.
					
					Finally, suppose $a < b$. If $2a \neq 2b$, then $a + b = b$ follows from the definition of a supertropical semiring. Suppose $2a = 2b$ (so $a$, $b$ are in the same fiber of $\nu$) while still $a < b$; then $b = \nu(b)$ by Lemma~\ref{LEMMA: Fibers of nu}, and consequently $a + b = \nu(b) = b$.
				\item
					If $b \cdot x$ is tangible, then $b$ (and $x$) must also be tangible, as $\nu{X}$ is an ideal. If $a + b = b$, then $a \cdot x + b \cdot x = (a + b) \cdot x = b \cdot x$. Using the previous item of the lemma, the result quickly follows. 
					
					Conversely, suppose $a \cdot x < b \cdot x$. Let us separate the cases depending on whether $a$ and $b$ are in the same fiber of $\nu$. If they are, \ie $2a = 2b$, then also $2ax = 2bx$, implying that $ax$ and $bx$ are also in the same fiber of $\nu$. Since $ax < bx$, Lemma~\ref{LEMMA: Fibers of nu} implies that $bx$ is a ghost. This leads to a contradiction. If $a$ and $b$ are not in the same fiber, then by the definition of the supertropical semiring $a + b \in \{a, b\}$. Assume first that $a+b=a$. Then $a \geq b$ and by monotonicity of multiplication $ax \geq bx$, which is a contradiction since $ax < bx$. If $a+b = b$, then according to item 1 of this lemma $a < b$ or $a \leq b \land b \in \nu{X}$. The second condition would force $b$ to be a ghost, which we know it is not. So $a<b$ as we tried to prove.
				\item
					Follows easily from the previous item.
				\item
					Suppose $x < y$. By the first item $x + y = y$, so $y = x + y = x + x + y$. Again by the first item we conclude $2x < y$.
					
					Clearly $2x < y$ implies the other two inequalities.
					
					Assume $2x < 2y$. By the first item $2y = 2x + 2y = 2(x + y)$, so $y$ and $x + y$ are in the same fiber. Since $y \leq x + y$, we have only two options. The first option is that $y = x + y$, in which case $x < y$, and we are done. The second option is that $y < x + y$ and $x + y = 2y$. 
			We cannot have $\nu(x) = \nu(y)$, since $2x < 2y$.
			Therefore $\nu(x) \neq \nu(y)$ and $2y = x+y \in \{x, y\}$. We have $x < 2y$, therefore $x + y = y$ and $x < y$ by the first item. 
				\item
					The ghost ideal $\nu{X}$ is linearly ordered, so we have $2x < 2y$, $2y < 2x$ or $2x = 2y$. The claim now follows from the previous item.
			\end{enumerate}
		\end{proof}
		
		\begin{lemma}\label{LEMMA: Tangible sum in a supertropical semiring}
			Let $X$ be a supertropical semiring.
			\begin{enumerate}
				\item
					$X$ is a bounded join-semilattice\footnote{A \df{bounded join-semilattice} is a partial order, in which we have joins (= suprema = least upper bounds) of all finite subsets. Equivalently, we need to have the join of the empty set (which is the smallest element) and of any pair of elements.}, with $\sup\emptyset = 0$ and for any $x, y \in X$
					\[\sup\{x, y\} = \begin{cases} x + y & \text{if $x \neq y$,} \\ x & \text{if $x = y$.} \end{cases}\]
					
				\item
				Supremum and $\nu$ commute, \ie for all $a_1, \ldots, a_n \in X$
				\[
				\nu(\sup\{a_1, \ldots, a_n\}) = \sup\{\nu(a_1), \ldots, \nu(a_n)\}.
				\] 
				\item
					Take any $n \in \NN$ and $a_1, \ldots, a_n \in X$. Let $s \dfeq a_1 + \ldots + a_n$ and ${M \dfeq \sup\{a_1, \ldots, a_n\}}$. Then $2s = 2M$ and
					\begin{align*}
						s \notin \nu{X} &\iff M \notin \nu{X} \land \exactlyone{i}{\{1, \ldots, n\}}{a_i = M} \\
						&\iff M \notin \nu{X} \land \exactlyone{i}{\{1, \ldots, n\}}{a_i = M = s}.
					\end{align*}
			\end{enumerate}
		\end{lemma}
		
		\begin{proof}
			\begin{enumerate}
				\item
					If $x = y$, then $\sup\{x, y\} = x = y$. If $2x \neq 2y$, then the upper bound $x + y$ is necessarily the least since it equals $x$ or $y$. Assume now $2x = 2y \dfeqrev a$ and $x \neq y$. Suppose $x = 2x$; then $x + y = 2x = x$, so again we are done. Likewise for $y = 2y$. The only remaining case is $x, y \in \nu^{-1}(a) \setminus \{a\}$. By the definition of supertropical semirings it follows that $x + y = a$. Suppose that $b \in X$ is an upper bound for $x$ and $y$ which is different from $a$. Then $x < b$ and $y < b$, since $x$ and $y$ are incomparable. We obtain that $a = x + y \leq 2b$. If $a = 2b$, then $b$ and $x$ lie in the same fiber of $\nu$ and they must be incomparable (being different from $a$) by Lemma~\ref{LEMMA: Fibers of nu}, contradicting $x < b$. Hence $2a = a < 2b$. If $b$ is ghost, then $b = 2b$ and $a < b$, as required. If $b$ is tangible, then $a < b$ by item 4 of Lemma~\ref{LEMMA: Strict inequality in supertropical semirings} we obtain that $x + y = \sup\{x, y\}$, in which case it follows from Lemma~\ref{LEMMA: Fibers of nu} that $x + y = a = \sup\{x, y\}$.
					\item
					The statement is obvious for $n=0$ and $n=1$. Let $n=2$. For $a_1 = a_2$, $\nu(\sup\{a_1, a_2\}) = \nu(a_1) = \sup \{ \nu(a_1), \nu(a_2)\}$ by item 1 of this lemma. For $a_1 \neq a_2$ and $\nu(a_1)\neq \nu(a_2)$, 
					\[
					\nu(\sup\{a_1, a_2\}) = \nu(a_1 + a_2) = \nu(a_1) + \nu(a_2)= \sup \{ \nu(a_1), \nu(a_2)\}.
					\] 
 For $a_1 \neq a_2$ and $\nu(a_1)= \nu(a_2)$, 
					\[
					\nu(\sup\{a_1, a_2\}) = \nu(a_1 + a_2) = \nu(a_1) + \nu(a_2)= 2\nu(a_1) =4 a_1 = 2 a_1 = \nu(a_1) = \sup \{ \nu(a_1), \nu(a_2)\}.
					\] 
					For higher $n$ inductively
					\[
					\begin{array}{lcl}
					\nu(\sup\{a_1, a_2, \ldots, a_n \}) &=& \nu(\sup\{a_1, \sup\{ a_2, \ldots, a_n \} \}) \\
					&=& \sup\{ \nu(a_1), \nu(\sup\{ a_2, \ldots, a_n \})\} \\
					&=& \sup\{ \nu(a_1), \sup\{ \nu(a_2), \ldots, \nu(a_n)\}\}\\
					&=& \sup\{ \nu(a_1), \nu(a_2), \ldots, \nu(a_n)\}.\\
					\end{array}
					\]
				\item
					Recall from the first item of Proposition~\ref{PROPOSITION: basic properties of supertropical semirings} that $\nu{X}$ is an idempotent semiring. Hence, by the fourth item of Proposition~\ref{PROPOSITION: Order of idempotent semirings} the sum in $\nu{X}$ is the same as supremum. Taking into account the previous item of this lemma, we get
					\[2M = 2\sup\{a_1, a_2, \ldots, a_n \} = \sup\{2a_1, 2a_2, \ldots, 2a_n \} = 2a_1 + 2a_2 + \ldots + 2a_n = 2s.\]
					
					We first prove that
					\[M \notin \nu{X} \land \exactlyone{i}{\{1, \ldots, n\}}{a_i = M = s} \iff M \notin \nu{X} \land \exactlyone{i}{\{1, \ldots, n\}}{a_i = M}.\]
					Direction~$\Rightarrow$ is clear. We now prove the reverse direction. If there exists exactly one $a_i$, equal to $M$, then $a_i > a_j$ for all other indices $j$. It follows $a_i + a_j = a_i$ by the first item of Lemma~\ref{LEMMA: Strict inequality in supertropical semirings}. Hence $a_i = s = M$.
					
					We now prove the equivalence
					\[M \notin \nu{X} \land \exactlyone{i}{\{1, \ldots, n\}}{a_i = M = s} \iff s \notin \nu{X}.\]
					The implication $\Rightarrow$ is obvious.
							
					It remains to show that
					\[s \notin \nu{X} \implies M \notin \nu{X} \land \exactlyone{i}{\{1, \ldots, n\}}{a_i = M = s}.\]
					Suppose $s \notin \nu{X}$. We always have $M \leq s \leq 2s = 2M$, so $s \in \nu^{-1}(2M) \setminus \{2M\}$. By Lemma~\ref{LEMMA: Fibers of nu} we get $s = M$, so also $M \notin \nu{X}$.
					
					By the previous item of this lemma $2M = \sup \{2a_1, \ldots, 2a_n\}$. Since $\nu{X}$ is linearly ordered (first item of Proposition~\ref{PROPOSITION: basic properties of supertropical semirings}), this supremum is attained. Let 
					\[
					I \dfeq \st{i \in \{1, \ldots, n\}}{2a_i = 2M};
					\]
					then ${2s = 2M = \sup \st{2a_i}{i \in I}}$.
					
					Let $i \in I$. Then $a_i \leq M < 2M = 2a_i$, so by Lemma~\ref{LEMMA: Fibers of nu} $a_i = M$, because they are comparable. For any $i, j \in I$ we have $a_i = a_j = M = s$, so if $i \neq j$, then $s < 2s = a_i + a_j \leq s$ --- a contradiction.
			\end{enumerate}
		\end{proof}
		
		As mentioned, the scope of Theorem~\ref{THEOREM: Elementarity in idempotent semirings} is limited when it comes to supertropical semirings since $2$ is in general not cancellable. In fact, it is cancellable if and only if the supertropical semiring is idempotent, as we get $2 \cdot 1 = 2 = 4 = 2 \cdot 2$.
		
		So what can we say about elementarity in supertropical semirings? The reasoning from the previous section does not work directly, as the symmetrization $\Sym$ no longer has all symmetric polynomial functions as fixed points; for example 
		\[
		\Sym[2](x y) = x y + y x = 2 x y \neq x y.
		\] 
		Nor can we redefine the symmetrization as the average (rather than the sum) over permutations because  natural numbers from $2$ onward (which are actually all equal to $2$) are not cancellable, much less invertible.
		
		As a way of getting around that, we introduce the \df{minimal symmetrization} of a pure monomial in the following way. Pick such $n, j, i_1, \ldots, i_j, d_1, \ldots, d_j \in \NN_{\geq 1}$ that ${i_1 < i_2 < \ldots < i_j \leq n}$ and that $d_k$s are pairwise unequal. Then 
		\[
		\Minsym\big((x_1 \ldots x_{i_1})^{d_1}(x_{i_{1}+1} \ldots x_{i_2})^{d_2} \ldots (x_{i_{j-1}+1} \ldots x_{i_j})^{d_j}\big)
		\]
		 is defined as the sum of terms $(x_{\sigma(1)} \ldots x_{\sigma(i_1)})^{d_1} \ldots (x_{\sigma(i_{j-1}+1)} \ldots x_{\sigma(i_j)})^{d_j}$ over all those permutations $\sigma \in S_n$ which do not put all the variable to the same power as in some already added summand. By the standard formula for permutations of multisets this gives us $\frac{n!}{i_1! (i_2 - i_1)! (i_3 - i_2)! \ldots (i_j - i_{j-1})! (n - i_j)!}$ terms.
		
		Observe that the elementary symmetric polynomials are special cases of minimal symmetrizations, namely $\el{k}(x_1, \ldots, x_n) = \Minsym(x_1 \ldots x_k)$.
		
		\begin{lemma}\label{LEMMA: Factoring elementary symmetric polynomials (supertropical)}
			Let $X$ be a supertropical semiring. Given $n, j, i_1, \ldots, i_j, d_1, \ldots, d_j \in \NN_{\geq 1}$ with $i_1 < i_2 < \ldots < i_j \leq n$ and $d_1 > d_2 > \ldots > d_j$, we have (at the level of functions)
			\[\Minsym\big((x_1 \ldots x_{i_1})^{d_1} \ldots (x_{i_{j-1}+1} \ldots x_{i_j})^{d_j}\big) =\]
			\[= \el{i_j}^{d_j}(x_1, \ldots, x_n) \cdot \Minsym\big((x_1 \ldots x_{i_1})^{d_1 - d_j} \ldots (x_{i_{j-2}+1} \ldots x_{i_{j-1}})^{d_{j-1} - d_j}\big)\]
			for all $x_1, \ldots, x_j \in X$.
		\end{lemma}
		
		\begin{proof}
			Clearly the statement holds if $d_j = 0$. Assume that $d_j > 0$.
			
			Let
			\begin{align*}
				&p(x_1, \ldots, x_n) \dfeq \Minsym\big((x_1 \ldots x_{i_1})^{d_1} \ldots (x_{i_{j-1}+1} \ldots x_{i_j})^{d_j}\big) \text{ \ and}\\
				&q(x_1, \ldots, x_n) \dfeq \Minsym\big((x_1 \ldots x_{i_1})^{d_1 - d_j} \ldots (x_{i_{j-2}+1} \ldots x_{i_{j-1}})^{d_{j-1} - d_j}\big).
			\end{align*}
			
			By Lemma~\ref{LEMMA: Factoring elementary symmetric polynomials} the statement is true for idempotent Frobenius semirings, in particular, it holds for $\nu{X}$. Thus for any $x_1, \ldots, x_n \in X$ 
			\[p\big(\nu(x_1), \ldots, \nu(x_n)\big) = \el{i_j}^{d_j}\big(\nu(x_1), \ldots, \nu(x_n)\big) \cdot q\big(\nu(x_1), \ldots, \nu(x_n)\big).\]
			Since $\nu$ is a semiring homomorphism,
			\[\nu\big(p(x_1, \ldots, x_n)\big) = \nu\big(\el{i_j}^{d_j}(x_1, \ldots, x_n) \cdot q(x_1, \ldots, x_n)\big).\]
			That is, the two sides, the equality of which we want to prove, are at the very least in the same fiber of $\nu$.
			
			As each summand of $p\big(x_1, \ldots, x_n\big)$ is also a summand of $\el{i_j}^{d_j}\big(x_1, \ldots, x_n\big) \cdot q\big(x_1, \ldots, x_n\big)$, we have $p\big(x_1, \ldots, x_n\big) \leq \el{i_j}^{d_j}\big(x_1, \ldots, x_n\big) \cdot q\big(x_1, \ldots, x_n\big)$.
			
			Suppose $p\big(x_1, \ldots, x_n\big) \in \nu{X}$; then it is the largest element in its $\nu$-fiber (Lemma~\ref{LEMMA: Fibers of nu}), so we have the equality we want.
			
			From here on suppose that $p\big(x_1, \ldots, x_n\big) \in X \setminus \nu{X}$. According to Lemma~\ref{LEMMA: Tangible sum in a supertropical semiring}, $p$ contains exactly one monomial $m$ which is strictly bigger that all the others at $(x_1, \ldots, x_n)$, and $p$ is equal to it. Since $m\big(x_1, \ldots, x_n\big) \in X \setminus \nu{X}$ and $\nu{X}$ is an ideal, all variables that appear in $m\big(x_1, \ldots, x_n\big)$ must be in $X \setminus \nu{X}$ as well.
			
			Let $\sigma \in S_n$ be such a permutation that $m = (x_{\sigma(1)} \ldots x_{\sigma(i_1)})^{d_1} \ldots (x_{\sigma(i_{j-1}+1)} \ldots x_{\sigma(i_j)})^{d_j}$. We claim that values of variables strictly decrease as we move from one block in $m$ to the next. More precisely, if $t \in \{1, \ldots, j-1\}$ and $u, v \in \NN$ are such that $1 \leq u \leq i_t < v \leq n$, then $x_{\sigma(u)} > x_{\sigma(v)}$. We prove this by eliminating all other options in part~\ref{LEMMA: Strict inequality in supertropical semirings:5} of Lemma~\ref{LEMMA: Strict inequality in supertropical semirings}.
			
			Let $m'$ be the monomial in $p$ which differs from $m$ only in having $x_{\sigma(u)}$ and $x_{\sigma(v)}$ switched. Since $m > m'$, $m > 2 m'$ and $m + m' = m$ in $(x_1, \ldots, x_n)$ by Lemma~\ref{LEMMA: Strict inequality in supertropical semirings}. Factor out the common part of $m$ and $m'$ to get
			\[m\big(x_1, \ldots, x_n\big) + m'\big(x_1, \ldots, x_n\big) = r\big(x_1, \ldots, x_n\big) \cdot \big(x_{\sigma(u)}^d + x_{\sigma(v)}^d\big)\]
			for a suitable monomial $r$ and $d \in \NN_{\geq 1}$.
			
			If $x_{\sigma(u)}$ and $x_{\sigma(v)}$ were in the same fiber of $\nu$, the same would hold for $x_{\sigma(u)}^d$ and $x_{\sigma(v)}^d$. Their sum would then be in $\nu{X}$, implying that ${m\big(x_1, \ldots, x_n\big) + m'\big(x_1, \ldots, x_n\big)}$ is in $\nu{X}$ as well, a contradiction.
			
			If $x_{\sigma(u)} \leq x_{\sigma(v)}$, then $m = m + m' \leq 2 m' < m$ in $(x_1, \ldots, x_n)$, likewise a contradiction. 
			
			We conclude that $x_{\sigma(u)} > x_{\sigma(v)}$.
			
			Take any $\pi, \rho \in S_n$, and consider the summand $s \dfeq (x_{\pi(1)} \ldots x_{\pi(j)})^{d_j} \cdot x_{\rho(1)}^{d_1 - d_j} \ldots x_{\rho(j-1)}^{d_{j-1} - d_j}$ from the right-hand side. We follow the proof of Lemma~\ref{LEMMA: Factoring elementary symmetric polynomials}. Let 
			\[
			I \dfeq \st{i \in \{1, \ldots, j\}}{\xsome{k}{\{1, \ldots, j\}}{\rho(i) = \pi(k)}}.
			\]
			 Since we can arbitrarily permute the variables in the product $x_{\pi(1)} \ldots x_{\pi(j)}$ without changing its value, we may assume without loss of generality that $\pi(i) = \rho(i)$ for all $i \in I$. Let $J \dfeq \{1, \ldots, j\} \setminus I$; then Frobenius property implies that
			\[x_{\pi(i)}^{d_j} \cdot x_{\rho(i)}^{d_i - d_j} \leq \sum_{k = 0}^{d_i} \binom{d_i}{k} x_{\pi(i)}^{k} x_{\rho(i)}^{d_i - k} = \big(x_{\pi(i)} + x_{\rho(i)}\big)^{d_i} = x_{\pi(i)}^{d_i} + x_{\rho(i)}^{d_i},\]
			for any $i \in J$. It follows that
			\[s \leq \prod_{i \in I} x_{\pi(i)}^{d_i} \cdot \prod_{i \in J} (x_{\pi(i)}^{d_i} + x_{\rho(i)}^{d_i}).\]
			If we use distributivity to fully expand this product, we see that each summand that we get also appears in $p\big(x_1, \ldots, x_n\big)$. If we get any monomial other than $m$, it does not change the sum (by Lemma~\ref{LEMMA: Strict inequality in supertropical semirings}), since it is strictly smaller than $m$. 
			
			We can also get $m$, but only in one way; once we prove this, the proof of the lemma is done. 
			The image of $\{1, \ldots, j\}$ under $\pi$ and $\rho$ has to be the same as under $\sigma$; in particular, $I = \{1, \ldots, j\}$, and $s = m$. Clearly, there is only one way to choose the appropriate term from $\el{i_j}^{d_j}\big(x_1, \ldots, x_n\big) = \Minsym(x_1 \ldots x_k)^{d_j} = \Minsym(x_1^{d_j} \ldots x_k^{d_j})$. As for the terms in $q(x_1, \ldots, x_n) = \Minsym\big((x_1 \ldots x_{i_1})^{d_1 - d_j} \ldots (x_{i_{j-2}+1} \ldots x_{i_{j-1}})^{d_{j-1} - d_j}\big)$, only one is given by $\rho$ which is the same as $\sigma$, up to permuting variables within blocks. The other terms are strictly smaller than that (recall from Lemma~\ref{LEMMA: Strict inequality in supertropical semirings} that multiplying with an element preserves $<$, as long as the result is tangible) since variables with smaller values (as shown above) appear in powers with larger exponents, and vice versa. Adding strictly smaller terms to $m$ does not change $m$.
		\end{proof}
		
		\begin{lemma}\label{LEMMA: Elementary factorization of an intrinsic symmetrization}
			Let $X$ be a supertropical semiring. Any minimal symmetrization of a pure monomial over $X$ is as a function equal to a product of elementary symmetric polynomials.
		\end{lemma}
		
		\begin{proof}
			Follows by induction from Lemma~\ref{LEMMA: Factoring elementary symmetric polynomials (supertropical)}.
		\end{proof}
		
		\begin{theorem}\label{THEOREM: Supertropical semirings are fully elementary}
			Any supertropical semiring is fully elementary.
		\end{theorem}
		
		\begin{proof}
			By definition we can write any symmetric polynomial as a linear combination of minimal symmetrizations of pure monomials. The claim then follows from Lemma~\ref{LEMMA: Elementary factorization of an intrinsic symmetrization}.
		\end{proof}

	\section{Elementarity in Symmetrized Semirings}\label{SECTION: Elementarity in Symmetrized Semirings}
	
		There exists a construction for semirings which `symmetrizes' them in a particular way~\cite{Mplus, tropicalindependence}. In this section we consider elementarity of such symmetrized semirings.
		
		Given a semiring $X$, its \df{quasisymmetrization} is defined as $\qsymm(X) \dfeq X \times X$ with operations
		\[(a', a'') + (b', b'') \dfeq (a' + b', a'' + b''),\]
		\[(a', a'') \cdot (b', b'') \dfeq (a' \cdot b' + a'' \cdot b'', a' \cdot b'' + a'' \cdot b').\]
		These make $\qsymm(X)$ into a semiring which is commutative/unital/upper-bound if $X$ is. The additive unit is $(0, 0)$, the multiplicative unit is $(1, 0)$. Note that $X$ embeds into $\qsymm(X)$, in the sense that $x \mapsto (x, 0)$ is an injective (unital) semiring homomorphism.
		
		Quasisymmetrizations are not particularly interesting when it comes to elementarity. In fact, we have the following proposition.
		
		\begin{proposition}\label{PROPOSITION: Elementarity of quasisymmetrizations}
			Let $X$ be an upper-bound unital commutative semiring. The following statements are equivalent.
			\begin{enumerate}
				\item
					$\qsymm(X)$ is fully elementary.
				\item
					$\qsymm(X)$ is Frobenius.
				\item
					$X$ is trivial (\ie $X = \{0\}$).
			\end{enumerate}
		\end{proposition}
		
		\begin{proof}
			\begin{itemize}
				\item\proven{$(1 \impl 2)$}
					By Theorem~\ref{THEOREM: Frobenius and 2-elementarity}.
				\item\proven{$(2 \impl 3)$}
					We have $(2, 0) = (1, 0)^2 + (0, 1)^2 = \big((1, 0) + (0, 1)\big)^2 = (1, 1)^2 = (2, 2)$. Hence $0 = 2$, and since $0 \leq 1 \leq 2$ and $X$ is upper-bound, we get $0 = 1$.
				\item\proven{$(3 \impl 1)$}
					Trivial.
			\end{itemize}
		\end{proof}
		
		However, a quasisymmetrization is generally just the first step towards constructing a new semiring. If all elements of $X$ are additively cancellable, the relation $\equ$, given by $(a', a'') \equ (b', b'') \dfeq a' + b'' = a'' + b'$, is a congruence on $\qsymm(X)$, \ie an equivalence relation which respects the semiring operations. Thus the quotient $\qsymm(X)/_\equ$ is a well-defined semiring. In fact, it is a ring, into which $X$ embeds via $x \mapsto [(x, 0)]$, and is the smallest one such in the suitable sense.
		
		This is a standard construction how to turn a semiring into a ring, but in case $X$ is not additively cancellable, a slight adjustment is required. The given $\equ$ is not transitive, and needs to be redefined to $(a', a'') \equ (b', b'') \dfeq \xsome{x}{X}{a' + b'' + x = a'' + b' + x}$. In that case the quotient $\qsymm(X)/_\equ$ is again a ring and we again have the canonical homomorphism $x \mapsto [(x, 0)]$, but this is no longer an embedding (it is not injective).
		
		To deal with this flaw, a different relation is considered in the context of tropical-like semirings (which are not additively cancellable)~\cite{gaubertmaxplus}. Recall that then $\equ$ is given by
		\[(a', a'') \equ (b', b'') \sepdfeq (a', a'') = (b', b'') \lor \big(a' \neq a'' \land b' \neq b'' \land a' + b'' = a'' + b'\big)\]
		for $(a', a''), (b', b'') \in \qsymm(X)$.
		
		This relation is not automatically a congruence, though; hence the following definition.
		
		\begin{definition}
			A semiring $X$ is \df{symmetrizable} when $\equ$ is a congruence, in which case $\symm(X) \dfeq \qsymm(X)/_\equ$ is a well-defined semiring. We call $\symm(X)$ the \df{symmetrization} of $X$.
		\end{definition}
				
		The classical example is $\mathbb{S}_{\max} \dfeq \symm(\RR_{\max})$, the symmetrization of the max-plus semiring~\cite{gaubertmaxplus}.
		
		It is easy to check that if $X$ is symmetrizable, then $x \mapsto [(x, 0)]$ is an injective (unital) semiring homomorphism from $X$ to $\symm(X)$.
		
		In this paper we limit ourselves to elementarity of upper-bound semirings, which also limits the consideration of symmetrizable semirings.
		
		\begin{lemma}\label{LEMMA: Upper-bound symmetrizable semirings}
			Let $X$ be a unital commutative semiring. The following statements are equivalent.
			\begin{enumerate}
				\item
					$X$ is symmetrizable and $\symm(X)$ is upper-bound.
				\item
					$X$ is idempotent, linearly ordered by its intrinsic order, and satisfies the following property: for all $a, b, x \in X$, if $a < b$, then $a \cdot x = 0$ or $a \cdot x < b \cdot x$.
			\end{enumerate}
		\end{lemma}
		
		\begin{proof}
			\begin{itemize}
				\item\proven{$(1 \impl 2)$}
					We have $[(1, 0)] + [(0, 1)] = [(1, 1)]$, so $[(1, 0)] \leq [(1, 1)]$. Suppose $1 \neq 2$ in $X$; then $[(1, 1)] + [(1, 0)] = [(2, 1)] = [(1, 0)]$, so $[(1, 1)] \leq [(1, 0)]$. But ${[(1, 0)] \neq [(1, 1)]}$, which contradicts the assumption that $\symm(X)$ is upper-bound. Thus $1 = 2$, \ie $X$ is idempotent.
					
					Now let $a, b \in X$ and suppose that neither $a \leq b$ nor $b \leq a$, meaning that ${a \neq a + b \neq b}$, and of course also $a \neq b$. Then $(a, b) \equ (a + b, b) \equ (a + b, a) \equ (b, a)$, so by transitivity $(a, b) \equ (b, a)$. This means that $(a, b) = (b, a)$ or $a = a + a = b + b = b$, a contradiction either way. It follows that $X$ is linearly ordered.
					
					Finally, let $a, b, x \in X$ with $a < b$. Then ${(0, b) \equ (a, b)}$, hence 
					\[
					(0, b) \cdot (x, 0) \equ (a, b) \cdot (x, 0),
					\] \ie $(0, b x) \equ (a x, b x)$. If $a x = 0$, we are done. Otherwise, $a x \neq b x$. Since $a \leq b$, $a x \leq b x$ and therefore $a x < b x$.
				\item\proven{$(2 \impl 1)$}
					It is clear from the definition that $\equ$ is reflexive and symmetric. To see that it is transitive, take any 
					\[
					(a', a''), (b', b''), (c', c'') \in \qsymm(X)
					\]
					 with $(a', a'') \equ (b', b'') \equ (c', c'')$. If any two of these pairs are equal, we are done. Otherwise we have $a' \neq a''$, $b' \neq b''$, $c' \neq c''$, $a' + b'' = a'' + b'$ and $b' + c'' = b'' + c'$. Assume $b' < b''$ (the case $b'' < b'$ is analogous). Since $b'' \leq b'' + a' = a'' + b'$, it follows $a'' > b'$, and then likewise $b'' > a'$, so $a'' = b''$. In the same way we get $b'' = c'' > b' < b'' > c'$. Hence $a' + c'' = a' + c'' + b' = a' + b'' + c' = a'' + b' + c' = a'' + c'$, which concludes the proof of transitivity.
					
					We have seen that $\equ$ is an equivalence relation. To conclude that it is a congruence, we still need to see that it respects addition and multiplication.
					
					Let $(a', a''), (b', b''), (c', c'') \in \qsymm(X)$ with $(a', a'') \equ (b', b'')$. If ${(a', a'') = (b', b'')}$, then $(a', a'') + (c', c'') \equ (b', b'') + (c', c'')$ and $(a', a'') \cdot (c', c'') \equ (b', b'') \cdot (c', c'')$. Suppose that $a' \neq a''$, $b' \neq b''$ and $a' + b'' = a'' + b'$. If $a' < a''$, then ${a' < a'' + b' = b''}$, so $b' \leq b''$, but $b' \neq b''$, so $b' < b''$. The case $a'' < a'$ is similar. Thus our assumption reduces to
					\[a' < a'' \land b' < b'' \land a'' = b'' \qquad \text{or} \qquad a'' < a' \land b'' < b' \land a' = b'.\]
					The two cases are analogous, so without loss of generality we restrict ourselves to the first one.
					
					Suppose that $(a', a'') + (c', c'') = (b', b'') + (c', c'')$; then 
					\[
					(a', a'') + (c', c'') \equ (b', b'') + (c', c'').
					\]
					 It cannot happen that $a'' + c'' \neq b'' + c''$ since $a'' = b''$. The only remaining case is that $a' + c' \neq b' + c'$. The pairs $(a', a'')$ and $(b', b'')$ appear symmetrically throughout, so assume without loss of generality that $a' + c' < b' + c'$. From here it follows that $a' < b'$ and $c' < b'$. To summarize: $a', c' < b' < b'' = a''$.
					
					Hence $a' + c' < b' + c' = b' < b'' \leq b'' + c'' = a'' + c''$, meaning that $a' + c' \neq a'' + c''$ and $b' + c' \neq b'' + c''$. Additionally, $a' + c' + b'' + c'' = a'' + c'' + b' + c'$ since $a' + b'' = a'' + b'$. We conclude $(a', a'') + (c', c'') \equ (b', b'') + (c', c'')$.
					
					As for products, we separate three cases.
					\begin{itemize}
						\item
							If $(a', a'') \cdot (c', c'') = (b', b'') \cdot (c', c'')$, then ${(a', a'') \cdot (c', c'') \equ (b', b'') \cdot (c', c'')}$.
						\item
							Suppose $a' c' + a'' c'' \neq b' c' + b'' c''$. Since $(a', a'')$ and $(b', b'')$ appear symmetrically, we assume without loss of generality that $a' c' + a'' c'' < b' c' + b'' c''$; since $a'' = b''$, it follows that $a' c' < b' c'$ and $a' < b'$. Also, 
							\[
							b'' c'' = a'' c'' \leq a' c' + a'' c'',
							\]
							 so $b' c' > b'' c''$, and therefore $c' > c''$ (since $b' < b''$). 
							
							Since $b' < b''$ and $b' c' \neq 0$ (because $b' c' > b'' c'' \geq 0$), we have $b' c' < b'' c'$. Thus $b' c'' + b'' c' = b'' c' > b' c' = b' c' + b'' c'' > a' c' + a'' c''$ and $a' c'' + a'' c' = a'' c' = b'' c'$, so $a' c' + a'' c'' \neq a' c'' + a'' c'$ and $b' c' + b'' c'' \neq b' c'' + b'' c'$.
							
							Also, 
							\[
							a' c' + a'' c'' + b' c'' + b'' c' = (a' + b'') c' + (a'' + b') c'' =
							\]
							\[
							= (a'' + b') c' + (a' + b'') c'' = a'' c' + a' c'' + b' c' + b'' c''.\] 
							In conclusion, $(a', a'') \cdot (c', c'') \equ (b', b'') \cdot (c', c'')$.
						\item
							The remaining case $a' c'' + a'' c' \neq b' c'' + b'' c'$ works the same as the previous one, just with $c'$ and $c''$ switched.
					\end{itemize}
					
					We have seen that for any $a, b, c \in \qsymm(X)$, if $a \equ b$, then $a + c \equ b + c$ and $a \cdot c \equ b \cdot c$. But then for any $a, b, c, d \in \qsymm(X)$ with $a \equ b$ and $c \equ d$ we have $a + c \equ b + c \equ b + d$ and $a \cdot c \equ b \cdot c \equ b \cdot d$. We already know that $\equ$ is transitive, so we conclude that $\equ$ is a congruence.
					
					$X$ is idempotent, therefore $\symm(X)$ is idempotent and is upper-bound by Proposition~\ref{PROPOSITION: Order of idempotent semirings}.
			\end{itemize}
		\end{proof}
		
		This lemma tells us, the symmetrizations of which semirings we should consider, but the result is the same as in the case of Proposition~\ref{PROPOSITION: Elementarity of quasisymmetrizations} (with essentially the same proof).
		
		\begin{proposition}\label{PROPOSITION: Elementarity of symmetrizations}
			Let the semiring $X$ satisfy the equivalent properties from Lemma~\ref{LEMMA: Upper-bound symmetrizable semirings}. The following statements are equivalent.
			\begin{enumerate}
				\item
					$\symm(X)$ is fully elementary.
				\item
					$\symm(X)$ is Frobenius.
				\item
					$X$ is trivial.
			\end{enumerate}
		\end{proposition}
		
		\begin{proof}
			\begin{itemize}
				\item\proven{$(1 \impl 2)$}
					By Theorem~\ref{THEOREM: Frobenius and 2-elementarity}.
				\item\proven{$(2 \impl 3)$}
					We have $[(2, 0)] = [(1, 0)]^2 + [(0, 1)]^2 = \big([(1, 0)] + [(0, 1)]\big)^2 = [(1, 1)]^2 = [(2, 2)]$, \ie $(2, 0) \equ (2, 2)$. Since $(2, 2)$ is equivalent only to itself, we conclude $0 = 2$, 
and since $0 \leq 1 \leq 2$ and $X$ is upper-bound, we get $0=1$.
				\item\proven{$(3 \impl 1)$}
					Trivial.
			\end{itemize}
		\end{proof}
	
	\section{Discussion}\label{SECTION: Conclusion and Questions}
	
		We studied elementarity --- the ability to express symmetric polynomials with elementary ones --- in upper-bound semirings. We have seen that $2$-elementarity is equivalent to the Frobenius property (Theorem~\ref{THEOREM: Frobenius and 2-elementarity}). We have seen that in idempotent semirings the Frobenius property is equivalent to full elementarity (Theorem~\ref{THEOREM: Elementarity in idempotent semirings} and Corollary~\ref{COROLLARY: Elementarity in idempotent semirings}).
		
		Furthermore, we proved that supertropical semirings (known to be Frobenius) are all fully elementary (Theorem~\ref{THEOREM: Supertropical semirings are fully elementary}).
		
		We gave a characterization for when the symmetrization of a semiring exists and yields an upper-bound semiring (Lemma~\ref{LEMMA: Upper-bound symmetrizable semirings}). We then showed that under these conditions no non-trivial symmetrization is Frobenius or fully elementary (Proposition~\ref{PROPOSITION: Elementarity of symmetrizations}).
		
		One of the goals of this paper was to answer the questions posed in~\cite{symtrop}. The following theorem proves that the symmetrized min-plus and the symmetrized max-plus semirings are not fully elementary.
		\begin{theorem}\label{THEOREM: Tropical-like semirings are fully elementary}
			The tropical semiring $\RR_{\min}$ and the max-plus semiring $\RR_{\max}$ are fully elementary. Their symmetrizations $\symm(\RR_{\min})$ and $\symm(\RR_{\max})$ are not.
		\end{theorem}
		
		\begin{proof}
			$\RR_{\min}$ and $\RR_{\max}$ are unital, commutative and idempotent. The intrinsic order on $\RR_{\max}$ is the usual one $\leq$ on $\RR$, and on $\RR_{\min}$ it is its opposite $\geq$; in both cases we get a linear order. Now apply Proposition~\ref{PROPOSITION: Linear idempotent semiring is Frobenius} and Corollary~\ref{COROLLARY: Elementarity in idempotent semirings} to get full elementarity.\footnote{Actually, $\RR_{\min}$ and $\RR_{\max}$ are also supertropical semirings, so we could have applied Theorem~\ref{THEOREM: Supertropical semirings are fully elementary} as well.}
			
			Their symmetrizations are not fully elementary by Proposition~\ref{PROPOSITION: Elementarity of symmetrizations}.
		\end{proof}
		
			Supetropical semirings, including the extended tropical semiring, are fully elementary.
					
		\begin{theorem}\label{THEOREM: Extended tropical semiring is fully elementary}
			The extended tropical semiring is fully elementary.
		\end{theorem}
		
		\begin{proof}
			Since the extended tropical semiring is a supertropical semiring, Theorem~\ref{THEOREM: Supertropical semirings are fully elementary} applies.
		\end{proof}
			
	The most general family of semirings, for which we managed to prove full elementarity, are supertropical semirings (Theorem~\ref{THEOREM: Supertropical semirings are fully elementary}). In other words, we have an analogue of the Fundamental Theorem of Symmetric Polynomials for supertropical semirings. Or rather, we have the existence part of this theorem. The uniqueness clearly does not hold; for example, in any Frobenius idempotent semiring the symmetric polynomial function $x^2 + y^2$ can be represented as $(x + y)^2$, as well as $(x + y)^2 + x y$. We do not know yet, whether a different notion of uniqueness could be defined that would allow a complete translation of the Fundamental Theorem of Symmetric Polynomials.
		
		Our findings raise a host of further question. While we have a characterization of $2$-elementarity for upper-bound semirings, our theorems about full elementarity were not as general.
		
		\begin{question}
			Does the Frobenius property characterize full elementarity in general upper-bound semirings? If not, what additional conditions are required?
		\end{question}
		
		Speaking of the Frobenius property, recall that it automatically holds in any linearly ordered upper-bound unital commutative idempotent semiring (Proposition~\ref{PROPOSITION: Linear idempotent semiring is Frobenius}), as well as in supertropical semirings which are very close to being linearly ordered. This leads to a question, in how general of semirings can we use linearity to prove Frobenius? In particular, note that (recall Lemma~\ref{LEMMA: 2 in Frobenius semirings}) we have the implications
		\[1 = 2 \implies \text{Frobenius} \implies 2 = 3\]
		in any linearly ordered upper-bound unital commutative semiring. The first implication does not reverse (for example, the extended tropical semiring is supertropical and thus Frobenius, but is not idempotent); what about the second one?
		
		\begin{question}
			Consider an upper-bound unital commutative semiring, linearly ordered by its intrinsic order. Is $2 = 3$ not just a necessary, but also a sufficient condition for such a semiring to be Frobenius?
		\end{question}
		
		Next, consider the results about semiring symmetrizations (Section~\ref{SECTION: Elementarity in Symmetrized Semirings}). They were very much negative, but maybe that is because the scope was too narrow --- we know from Lemma~\ref{LEMMA: Upper-bound symmetrizable semirings} that if the symmetrization $\symm(X)$ exists and is upper-bound, the semiring $X$ is necessarily linearly ordered (among other things). An adaptation of the symmetrizing relation was considered in~\cite[Definition~2.2.9]{gaubert92a}.
		
		\begin{question}
			Does the symmetrizing relation from~\cite[Definition~2.2.9]{gaubert92a}, or any other reasonable adaptation of the symmetrizing relation, lead to better behaved symmetrized semirings in terms of elementarity?
		\end{question}
		
		We return to the notion of symmetric polynomial functions. As stated in Remark~\ref{REMARK: variants of symmetry}, there are two reasonable notions for a polynomial function to be symmetric. For our results we needed the `syntactic' one (except in the case of idempotent semirings, where we know that the two notions coincide), but do our results hold also if we assume the `semantic' definition?
		
		\begin{question}\label{QUESTION: match between syntactic and semantic symmetry}
			Is it the case for every upper-bound semiring that a polynomial function, invariant under permutations of variables, is necessarily represented by a symmetric polynomial? If not, is it true at least for supertropical semirings?
		\end{question}
		
	Finally, we did not discuss the computational complexity of our method for reduction to elementary symmetric functions. It would be interesting to do it and compare it with the known complexity results over fields.

	\printbibliography

\end{document}